\documentclass[12pt,reqno]{amsart}
\usepackage[ansinew]{inputenc}
\usepackage{amsrefs}
\usepackage{amsfonts}
\usepackage{amsmath}

\newcommand{\romanletters}
{\renewcommand{\theenumi}{\roman{enumi}}
\renewcommand{\labelenumi}{\textup{(}\theenumi\textup{)}}}

\newcommand{\letters}
{\renewcommand{\theenumi}{\alph{enumi}}
\renewcommand{\labelenumi}{\textup{(}\theenumi\textup{)}}}
\newcommand{\itemref}[1]{\eqref{#1}}

\newcommand{\N}{\mathbb N}

\newcommand{\C}{\mathbb C}
\newcommand{\Q}{\mathbb Q}
\newcommand{\ga}{\gamma}
\DeclareMathOperator{\F}{\mathcal{F}_\ga}
\newcommand{\hr}{\mathcal{H}_{r}(\C)}
\newcommand{\dhr}{\mathcal{H^\prime}_{\!r}(\C)}
\newcommand{\ear}{\mathrm{Exp}_{r,a}(\C)}
\newcommand{\e}{\mathrm{Exp}_r(\C)}
\newtheorem{theorem}{Theorem}[section]
\newtheorem{lemma}[theorem]{Lemma}
\newtheorem{proposition}[theorem]{Proposition}
\newtheorem{corollary}[theorem]{Corollary}

\theoremstyle{definition}
\newtheorem{definition}[theorem]{Definition}
\theoremstyle{definition}
\newtheorem{remark}[theorem]{Remark}

\begin{document}
\title[\hfilneg { }\hfil]{On harmonic analysis associated with the Hyper-Bessel operator on the complex plane }

\author{M. S. BEN HAMMOUDA }
\address{Med Saber Ben Hammouda \newline
Department of Mathematics, Faculty of Applied Sciences, Umm-Al-Qura University, AlAbdiya
P.O. Box 673, Makkah 21599, Saudi Arabia.}
\email{Mshamouda@uqu.edu.sa}

\author{L. BENNASR}
\address{Lassad Bennasr \ \newline Institut Préparatoire aux Etudes d'Ingénieur d'Elmanar, Campus
Universitaire D\'el Menzah, 1060 Tunis, Tunisie}
\email{lassad.bennasr@yahoo.fr}

\author{Ahmed Fitouhi}
\address{Ahmed Fitouhi \newline
 D\'epartement de Mat\'ematiques,  Facult\'e Des Sciences de Tunis, Campus
Universitaire D\'el Menzah, Tunis 1060, Tunisie  }
\email{Ahmed.Fitouhi@fst.rnu.tn}

\subjclass[2010]{44A05; 44A35; 46F12; 47A16}
\keywords{Hyper-Bessel operator, Convolution, Generalized Fourier transform, Hypercyclic and chaotic operator, Entire functions. \hfill\break}
\begin{abstract} We investigate the harmonic analysis associated with the  hyper-Bessel operator on $\mathbb{C}$, and we prove the chaotic character of the related convolution operators.
\end{abstract}
\maketitle
\section{Introduction}\label{intro} In connection with the generalization of the classical Bessel functions theory, M. I. Klyuchantsev~\cite{M.I.Klyu} introduced the hyper-Bessel operator
\begin{equation}\label{eq : Br1}
  B_{r}:= \frac{1}{z^{r-1}}\prod_{i=1}^{r-1}\big( z\frac{d}{dz}+({ r \gamma_{i} + 1})\big)\frac{d}{dz},
\end{equation}
where $r$ is an integer such that $r\geq2$, and $\ga=(\ga_1,\dots,\ga_{r-1})$ is a vector index having $r-1$ reals components. The operator $B_r$ contains as particular cases:
\begin{itemize}
\item[$\bullet$]The operator $\dfrac{d^r}{dz^r}$ when $\gamma_k=-1+\dfrac{k}{r},\quad k\in\{1,\dots,r-1\}$;
\item[$\bullet$] The classical Bessel operator of the second order
\begin{equation}\label{b2}
  B_{2}:= \frac{d^2}{dz^2} +\frac{(2\gamma  + 1)}{z}\frac{d}{dz},\qquad \gamma\geq-\frac{1}{2};
\end{equation}
\item[$\bullet$] The operator $B_{3} $ studied in \cite{fmren} and \cite{FITsab1}:
\begin{equation}
B_{3}:=\frac{d^3}{dz^3} + \frac{(3\,\nu)}{z}\,\frac{d^2}{dz^2}-\frac{(3\,\nu)}{z^2}\,\frac{d}{dz},\qquad\gamma_1 = -2/3,\,\gamma_2=\nu-1/3.
 \label{equ1}
\end{equation}
\end{itemize}
In \cite{FIT2} the authors studied the polynomial expansion for the solutions of the heat equation associated with the operator $B_r$, and very recently, in \cite{Dunkl-Bouz} the operator $B_r$ found its interpretation in the theory of special functions associated with complex reflection groups. A more general version of the hyper-Bessel operator is intensively studied by I. Dimovsky and V. Kyriakova in connexion with fractional calulus and operational calculi (see \cite{kiryakova1} and references therein). However, harmonic analysis associated the the hyper-Bessel operator is not yet developed.

In this paper, we are concerned with establishing some elements of harmonic analysis associated with the operator $B_r$ in the complex plane. Namely, we introduce the generalized Fourier transform and prove a generalized Paley-Wiener theorem.  Furthermore, we prove the chaotic character of the related convolution operators on some space of entire functions.

We recall that a continuous linear operator $T$ from a Fréchet space $X$ into itself is said chaotic in the sense of Devaney if
\begin{enumerate}\letters
\item $T$ is hypercyclic. That is, there exists $x\in X$ (that is called hypercyclic vector of $T$) such that its orbit $\left\{x,\,Tx,\,T^2x,\dots\right\}$ is a dense subset of $X$,
\item the set $Per(T)=\{x\in X;\, T^nx=x,\quad \text{for some integer $n$}\}$ of periodic points of $T$ is dense in $X$.
\end{enumerate}
In 1991, G. Godefroy and J. Shapiro \cite{gode} generalized classical works of Birkhoff~\cite{birk} and MacLane~\cite{macl}
showing respectively that the operators of translation and differentiation on $\mathcal{H}(\C)$, the space of all entire functions on $\C$, are hypercyclic, to show that every convolution operator on the space $\mathcal{H}(\C^N)$
that is not a scalar multiple of the identity is chaotic. Since then, the chaotic character of convolution operators have been the object of extensive study.
See for instance \cites{aron,Betancor1,Betancor2,Betancor3,Bonet,kim} for some recent works and the monographs \cite{bay} and \cite{linchao} for a general account of the theory.

This paper is organized as follows:
 \textrm{Section}~\ref{s2} is devoted to some notations and backgrounds including the Bessel function of vector index and the space $\hr$ of all $r$-even entire functions on $\C$. In Section~\ref{s3} we introduce the generalized Fourier transform and then prove a generalized
 Paley-Wiener theorem. In Section~\ref{s4} we study the generalized translation and generalized convolution on the space $\hr$ and its dual space, we also deals with the surjectivity of convolution operators. In Section~\ref{s5} we establish several characterizations of the continuous linear mappings from $\hr$ into itself that commute with the generalized translations, as a consequence, we prove the analogue of the results of Godefroy and Shapiro for the operator $B_r$.
\section{Preliminaries}\label{s2}
Throughout this paper we denote by $\N$ the nonnegative integers $\{0,1,\cdots\}$. We fix $r\in\N$ such that $r\geq 2$ and a vector index $\ga=(\ga_1,\dots,\ga_{r-1})$ having $(r-1)$ real components satisfying
\begin{equation*}\label{cond : gamma}
\gamma_k \ge - 1 + \frac{k}{r},\qquad\text{for all $ k \in\{1,...,r - 1\}$}.
\end{equation*}
The Bessel operator of $r$-order given by \eqref{eq : Br1} can also be written in the form
\begin{equation}\label{def : br}
B_r=\frac{d^r}{dz^r}+\frac{a_1}{z}\frac{d^{r-1}}{dz^{r-1}}+\ldots+\frac{a_{r-1}}{z^{r-1}}\frac{d}{dz},
\end{equation}
where
\begin{equation}\label{eq : ak}
a_{r - k} = \frac{1}{(k - 1)!}\sum\limits_{j = 1}^k {( - 1)^{k - j}
\binom{j - 1}{k - 1} } \prod\limits_{i = 1}^{r - 1} {(r\gamma_i +j)},
\end{equation}
for every $k\in\{1,\dots, r-1\}$.

Let $w=e^{{2i\pi}/{r}}$ be the root of the unit. A function $u\colon\C\to\C$ is called $r$-even if, $u(w^k z)=u(z)$ for all $ k\in\{1,\ldots r-1\}$.
We denote by $\hr$ the space of all $r$-even entire functions on $\C$ endowed with the topology of the uniform convergence on compact subsets of $\C$.
Recall that this topology is generated by the semi-norms
\begin{equation}
\|u\|_R:=\sup_{|z|\leq R}|u(z)|, \qquad u\in\hr,\,\, R>0.
\end{equation}
Thus, $\hr$ is a Fréchet space.

If  $u\in\hr$, then using integration by parts and taking into account that $u^\prime(0)=\dots=u^{(r-1)}(0)=~0$,  we conclude from \eqref{def : br} that
\begin{equation}\label{eq: Br integ}
B_ru(z)=u^{(r)}(z)+\sum_{k=1}^{r-1}\frac{a_k}{(k-1)!}\int_0^1(1-t)^{k-1}u^{(r)}(tz)dt,
\end{equation}
for all $z\neq0$. Hence, by the analyticity theorem, we can extend the function $B_ru$ to the whole complex plane, so that $B_ru\in\hr$.

The normalized Bessel function of the vector index $\gamma$ is the function $j_\gamma\colon\C\to\C$ given by
\begin{equation}\label{serie : j}
{j}_{\gamma}(z):= \sum_{n=0}^{\infty}\frac{(-1)^{n}z^{rn}}{\alpha_{rn}(\gamma)},\qquad z\in\C,
\end{equation}
where
\begin{equation}\label{const : alpha}
{\alpha_{rn}(\gamma)}= (r)^{rn}\,n! \,\prod_{i=1}^{r-1}
\frac{\Gamma(\gamma_{i} + n+1)}{\Gamma(\gamma_{i} +1)},\qquad n\in\N.
\end{equation}
See~\cites{FIT2,M.I.Klyu} and \cite{kiryakova1}*{Chap~3\&4} for more details.

For $\lambda\in\C$, define ${j}_{\gamma}(\lambda\cdot)\colon\C\to\C$ to be the function given by
\begin{equation}\label{eq : jlambda}
[j_\gamma(\lambda\cdot)](z)={j}_{\gamma}(\lambda z),\qquad\text{for all $z\in\C$}.
\end{equation}
\begin{proposition}\label{pr : jgamma egenfunction}(see \cite{FIT2})For every $\lambda\in\C$, the function $j_\gamma(\lambda\cdot)$ is the unique solution in $\hr$ of the system
\begin{equation}
\begin{cases}
B_{r}u(z)&=-\lambda^{r} u(z);\\
u(0)&=1;\\
D_{z}^{k}u(0)&=0, \text{ for all $k\in\{1,\dots,r-1\}$}.
\end{cases} \label{equ2q}
\end{equation}
\end{proposition}
By induction on the integer $n$ we can see that $\alpha_{rn}(\gamma) \geq (rn)!$ for all $n\in\N$. Hence, by \eqref{serie : j}, the function $j_\gamma$ satisfies the inequality
\begin{equation}\label{eq : est j}
\big|{{j}}_{\gamma}(z )\big|\leq G_\ga(|z|)\leq \, e^{|z| },\qquad\text{for all $z\in\C$},
\end{equation}
where
\begin{equation}\label{def : G}
G_{\ga}(z)=j_{\ga}(e^{i\pi/r}z)=\sum_{n=0}^{+\infty}\frac{z^{rn}}{\alpha_{rn}(\ga)}.
\end{equation}

\section{The generalized Fourier transform}\label{s3}
Let $\dhr$ denote the strong dual space of the space $\hr$ and let $\e$ denote the space of all $r$-even entire functions of exponential type. That is the space of all\ $u\in\hr$ for which
there exists $a>0$ such that
\begin{equation}\label{eq : norm pa}
P_{a}(u) = \sup_{z\in \C}|u(z)| \,e^{-a|z|}<\infty.
\end{equation}
For $a>0$, let $\ear$ denote the Banach space of all $u\in\e$ satisfying \eqref{eq : norm pa} provided with the norm $P_a$. Thus,
\begin{equation*}
\e=\bigcup_{a>0}\ear.
\end{equation*}
We provide $\e$ with the natural locally convex inductive limit topology.
\begin{lemma}\label{pr : est bn} Let $u(z)=\sum_{n=0}^{+\infty}\frac{b_n}{\alpha_{rn}(\ga)}z^{rn}\in\hr$. Then $u\in \e$ if and only if, there exist $a$, $C>0$ such that $|b_n|\leq C a^{rn}$, for all\ $n\in\N$.
\end{lemma}
\begin{proof}From \cite{Rubel}*{Page 43} we know that $u\in\e$ if and only if
\begin{equation*}
\limsup_{n\to\infty}\left|\frac{(rn)!b_n}{\alpha_{rn}(\ga)}\right|^{\frac{1}{rn}}<\infty.
\end{equation*}
Using Stirling formula, we can check that
\begin{equation}\label{eq : limsup}
\limsup_{n\to\infty}\left|\frac{(nr)!}{\alpha_{rn}(\ga)}\right|^{\frac{1}{n}}=1.
\end{equation}
This implies that $u\in\e$ if and only if $\limsup_{n\to \infty}\left|b_n\right|^{\frac{1}{rn}}<\infty$, and the conclusion of the lemma follows.
\end{proof}
\begin{lemma}
Let $v(z)=\sum_{n=0}^{+\infty}\frac{b_n}{\alpha_{rn}(\ga)}z^{rn}\in\e$. For $u(z)=\sum_{n=0}^{+\infty}a_nz^{rn}\in\hr$, let
\begin{equation}\label{eq : Tu in dh}
T_v(u)=\sum_{n=0}^{+\infty}a_nb_n.
\end{equation}
Then, the series in \eqref{eq : Tu in dh} converges absolutely and defines $T_v\colon u\to T_v(u)$ as an element of $\dhr$.
\end{lemma}
\begin{proof}By Lemma~\eqref{pr : est bn} there exist $C$, $a>0$ such that $|b_n|\leq C a^{rn}$ for all $n\in\N$. On the other hand, by the Cauchy estimate, for $u(z)=\sum_{n=0}^{+\infty}a_nz^{rn}\in\hr$, we have $|a_n|\leq (2a)^{-rn}\|u\|_{2a}$, for all $n\in\N$. This implies that the series in \eqref{eq : Tu in dh} converges absolutely and $|T_v(u)|\leq 2C\|u\|_{2a}$ and hence, $T_v\in\dhr$.
\end{proof}
\begin{definition} For $T\in\dhr$, we define the generalized Fourier transform of $T$ to be the function $\F(T)\colon\C\to\C$ given by
\begin{equation}
\F(T)(z)=\langle T(w), j_\ga(wz)\rangle,\qquad z\in\C,
\end{equation}
where $j_\ga$ is given by \eqref{serie : j}.
\end{definition}

\begin{theorem}[Palely-Wiener Theorem]\label{th : pw} The transform $\F$ is a topological isomorphism from $\dhr$ onto $\e$.
\end{theorem}
\begin{proof}First, we prove that if $T \in \dhr$, then $\F(T)\in\e$. Let $T\in\dhr$. Since the series in \eqref{serie : j} converges in $\hr$, it follows that
\begin{equation}\label{eq : F serie}
\F(T)(z) = \langle T(w),\, {{j}}_{\gamma}(w z) \rangle=\sum_{n=0}^{+\infty} (-1)^{n} \frac{\langle T(w),w^{nr}\rangle}{\alpha_{rn}(\gamma)} z^{nr},
\end{equation}
for all $z\in\C$. Hence, $\F(T)\in\hr$. On the other hand, the continuity of $T$ infers that there exist $a$, $C>0$ such that
$$|\langle T, u\rangle|\leq C\|u\|_a,\qquad\text{ for all $u\in\hr$},$$
which implies that $|\langle T(w), w^{nr}\rangle|\leq C a^{nr}$ for all $n\in\N$. It follows from \eqref{eq : F serie} and Proposition~\ref{pr : est bn} that $\F(T)\in\e$. %

Next, we prove that $\F$ is a continuous mapping from $\dhr$ into $\e$. Since $\dhr$ is a bornological space, it is sufficient to show that $\F(B)$ is a bounded set in $\e$ whenever $B$ is a bounded set of $\dhr$. Assume that $B$ is a bounded subset of $\dhr$. Since $\dhr$ is barreled, it follows from the uniform boundedness principle that there exist $C$, $a> 0$ such that $|\langle T, u\rangle|\leq C\|u\|_a$ for all $T\in B$ and all $u\in\hr$. Using \eqref{eq : est j}, we conclude that if $T\in B$, then $|\F(T)(z)|\leq C e^{a|z|}$ for all $z\in\C$, and hence $P_a(\F(T))\leq C$. This implies that $\F(B)$ is bounded in $\ear$ and hence bounded in $\e$.

To see that the transform \;$\F$\;is one to one, assume that\;$T\in \hr$\;is such that\;$\F(T)(z)=0$\;for all $z\in \C$. Then, from \eqref{eq : F serie} %
we conclude that $\langle T(w),w^{rn}\rangle =0$\; for all \;$n\in\N$. So, if $u(z)=\sum_{n=0}^{+\infty}b_nz^{rn}\in\hr$, then
$\langle T,u\rangle =\sum_{n=0}^{+\infty}b_n\langle T, w^{rn}\rangle=0$. Thus, $\F$ is one to one.

For the surjectivity of\;$\F$, let $v(z)=\sum_{n=0}^{+\infty}\frac{b_n}{\alpha_{rn}(\gamma )}z^{rn}\in\e$. Define $\tilde{v}(z)=\sum_{n=0}^{+\infty}\frac{(-1)^nb_n}{\alpha_{rn}(\gamma )}z^{rn}$. Then $\F(T_{\tilde{v}})=v$, where $T_{\tilde{v}}$ is given by~\eqref{eq : Tu in dh}.

Finally, the continuity of $(\F)^{-1}$ follows from the open mapping theorem \cite{vogt}*{Theorem~24.30}.
\end{proof}
\begin{corollary}\label{cor : dens}
Suppose $\Lambda$ is any subset of\ $\C$ with a limit point in $\C$, and let $J(\Lambda)$ be the linear span of the functions %
$j_{\gamma}(\lambda\cdot)$  with $\lambda\in\Lambda$. Then $J(\Lambda)$ is dense in $\hr$.
\end{corollary}
\begin{proof}
Suppose that $T\in\dhr$ is such that $\langle T(w), j_\gamma(\lambda w)\rangle=0$, for all $\lambda\in\Lambda$. Then $\F(T)(\lambda)=0$ for all $\lambda\in\Lambda$.
Since, by \rm{Theorem}~\ref{th : pw}, $\F(T)$ is entire on $\C$ and $\Lambda$ has a limit point in $\C$, it follows that $\F(T)$ vanishes on $\C$. Hence,
\rm{Theorem}~\ref{th : pw} also implies that $T$ vanishes on $\hr$, so by the Hahn-Banach theorem, $J(\Lambda)$ is dense in $\hr$.
\end{proof}
\section{The generalized translation and generalized convolution}\label{s4}
In this section we study the generalized translation and generalized convolution associated with the operator $B_r$.
\subsection{The generalized translation} We begin with the following useful lemma.
\begin{lemma}\label{lem : ineq br} If $u\in\hr$ and $R>0$, then
\begin{equation}
\|B_r^nu\|_R \leq \frac{M^n(nr)!}{R^{nr}}\|u\|_{2R}\label{eq : ineq br},
\end{equation}
for all $n\in\N$, where
\begin{equation}
M=1+\sum_{k=1}^{r-1}\frac{|a_k|}{k!}.
\end{equation}
\end{lemma}
\begin{proof}First, we prove by induction on $n$ that
\begin{equation}
\|B_r^nu\|_R\leq M^n\|u^{(nr)}\|_R,\qquad\text{for all $n\in\N$}.\label{eq : ineq br2}
\end{equation}
The result is trivially true for $n=0$. Assume that \eqref{eq : ineq br2} is satisfied for the integer $n$. Then,
\begin{equation}
\|(B_r)^{n+1}u\|_R=\|B_r^n(B_ru)\|_R\leq M^n\|(B_ru)^{(nr)}\|_R.
\end{equation}
Using \eqref{eq: Br integ} we conclude that
\begin{equation*}
(B_ru)^{(nr)}(z)=u^{((n+1)r)}(z)+\sum_{k=1}^{r-1}\frac{a_k}{(k-1)!}\int_0^1(1-t)^{k-1}t^{nr}u^{((n+1)r)}(tz)dt,
\end{equation*}
for all $z\in\C$. This implies that $\|(B_ru)^{(nr)}\|_R\leq M\|u^{(n+1)r}\|_R$. Hence,
$$\|(B_r)^{n+1}u\|_R\leq M^{n+1}\|u^{(n+1)r}\|_R,$$
and the proof of \eqref{eq : ineq br2} is complete.

Now, according to the Cauchy integral formula, we have
\begin{equation}\label{eq : ineq cauchy}
\left\Vert u^{(k)}\right\Vert _{R}\leq\frac{k!}{R^{k}}\left\Vert u\right\Vert _{2R},\qquad\text{for all $k\in\N$}.
\end{equation}
This together with \eqref{eq : ineq br2} yield \eqref{eq : ineq br}.
\end{proof}
\begin{corollary}\label{prop : cont br}
The operator $B_{r}$ is continuous from $\hr$ into~{itself.}
\end{corollary}

\begin{proof}
The result follows immediately from \eqref{eq : ineq br} by taking $n=1$.
\end{proof}
Following J. Delsarte~\cite{dels}, we define the generalized translation operator as follows
\begin{equation}\label{serie : trans}
(T^{\ga}_{z}u)(w)=\sum_{n=0}^{+\infty} \frac{w^{rn}}{\alpha_{rn}(\gamma)} B_{r}^{n}u(z),\qquad w,z\in\C,~u\in\hr.
\end{equation}
Note that by Proposition~\ref{pr : jgamma egenfunction}, we have the product formula
\begin{equation}\label{for : prod}
j_\ga(\lambda z) j_\ga(\lambda w)=T^{\ga}_{z}(j_\ga(\lambda.))(w),\qquad w,z\in\C.
\end{equation}
\begin{proposition}\label{prop : conti trans}The series in \eqref{serie : trans} converges on compact subsets of $\C$ and defines $T_z^\ga$ as a continuous linear operator from $\hr$ into itself, for evry $z\in\C$.
\end{proposition}
\begin{proof}let $R$, $R'>0$ and let $u\in\hr$. It follows from Lemma~\ref{lem : ineq br} that for all $w$, $z\in\C$ such that $|w|\leq R$ and $|z|\leq R'$ we have
\begin{equation}\label{eq : est term serie trans}
|\frac{w^{rn}}{\alpha_{rn}(\gamma)} B_{r}^{n}u(z)|\leq\frac{R^{nr}M^n(nr)!}{(R')^{nr}\alpha_{rn}(\ga)}\|u\|_{2R^\prime},\qquad n\in\N.
\end{equation}
Since
\begin{equation}
\limsup_{n\to\infty}\left|\frac{(nr)!}{\alpha_{rn}(\ga)}\right|^{\frac{1}{n}}=1,
\end{equation}
for sufficiently large $R'$, we have $(R/R^\prime)^rM<1$, and then, the series in \eqref{serie : trans} converges uniformly on the closed polydisk
 \begin{equation}
 \left\{(w,z)\in\C\times\C\mid |w|\leq R,\,|z|\leq R'\right\}.
 \end{equation}
Thus, the series converges uniformly on compact subsets of $\C\times\C$. Now, if we fix $z\in\C$ and we chose $R'$ sufficiently large so
that $|z|\leq R'$ and  $(R/R^\prime)^rM<1$, from \eqref{eq : est term serie trans} we conclude that there exists $C>0$ such that
\begin{equation}
\|T_z^\ga u\|_R\leq C\|u\|_{2R'},\qquad\text{for all $u\in\hr$}.
\end{equation}
Since $T_z^\ga$ is linear, it is continuous.
\end{proof}
\begin{remark} We define a generalized addition formula associated with the operator $B_r$, for $z,w\in\C$ and $n\in\C$, as follows
\begin{equation}
\left(z\oplus_\ga w\right)^{rn}:=\sum_{k = 0}^n \binom{\alpha_{rn}}{\alpha_{rk}} \,w^{rk}z^{r(n - k)}\label{eq : addition},
\end{equation}
where $$\dbinom{\alpha_{rn}}{\alpha_{rk}}:=\dfrac{\alpha_{rn}}{\alpha_{rk}\,\alpha_{r(n-k)}}$$ is the generalized binomial.
It is easy to check that for all $n\in\N$ and all $z,w\in\C$ we have
\begin{equation}
\left(z\oplus_\ga w\right)^{rn}=z^{rn}\,_{r}F_{r-1}\left[\begin{array}{c}-n, -(n + \gamma_{i}) \\[2pt]\gamma_{i}+1\end{array} \bigg|(-\frac{w}{z})^r\right],
\end{equation}
where $_{r}F_{r-1}$ is the hypergeometric function (see \cite{AA}). This addition in terms of hypergeometric functions is an analogue of %
the addition for Bessel functions presented by Bochner \cite{Boch55} and F. M. Cholewinski and J. A. Reneke \cite{fmren} (see also \cite{FITsab1}).
For $u(t)=\sum_{n=0}^{+\infty}b_nt^{nr}\in\hr$, define
\begin{equation}
u(z\oplus_{\ga}w)=\sum_{n=0}^{+\infty}b_n\left(z\oplus_\ga w\right)^{rn},\qquad z,w\in\C.
\end{equation}
Then, it is easy to check that $T_z^\ga u(w)=u(z\oplus_\ga w)$ for all $z,w\in\C$.
\end{remark}
In the following proposition, we give some properties of the generalized translation operators $T_z^\gamma$ for $z\in\C$.
\begin{proposition}\label{prop : prop trans} For $u\in\hr$ and  $z,w\in \C$ we have
\begin{enumerate}\romanletters
\item\label{a} $T^{\gamma}_{0}u(z)=u(z)$.
\item\label{b} $T^{\gamma}_{z}u(w)=T^{\gamma}_{w}u(z)$.
\item\label{d}$B_{r}T^{\gamma}_{z}u=T^{\gamma}_{z}B_{r}u$.
\item\label{e}$T^{\gamma}_{z} \circ T^{\gamma}_{w}u =T^{\gamma}_{w} \circ T^{\gamma}_{z}u$.
\end{enumerate}
\end{proposition}
\begin{proof}These properties are satisfied when $u=j_\ga(\lambda \cdot)$ where $\lambda\in\C$ and hence, by Corollary~\ref{cor : dens}, are satisfied by all $u\in\hr$.
\end{proof}
\begin{proposition}
Let $u \in \hr$. Then, the mapping  $F_{u}\colon z\mapsto T^\ga_zu$ is continuous from $\C$ into $\hr$.
\end{proposition}
\begin{proof}By Proposition~\ref{prop : prop trans}\itemref{b}, we can write
\begin{equation}
T^\ga_zu(w)=T^\ga_wu(z)=\sum_{n=0}^{+\infty} \frac{z^{rn}}{\alpha_{rn}(\gamma)} B_{r}^{n}u(w),\qquad\text{for all $z,w\in\C$},
\end{equation}
or
\begin{equation}\label{eq : serie trans 1}
T^\ga_zu=\sum_{n=0}^{+\infty} \frac{z^{rn}}{\alpha_{rn}(\gamma)} B_{r}^{n}u,\qquad\text{for all $z\in\C$}.
\end{equation}
Thus, the mapping  $F_u$ possesses a power series expansion with coefficients in the Fréchet space $\hr$ and radius of convergence %
$R=+\infty$ (see \cite{vectpower}*{Apend. A}). So it is continuous on~$\C$.
\end{proof}
\subsection{The generalized convolution}
\begin{definition}\label{def : conv}
For $T\in\dhr$ and $u \in \hr$ we define the generalized convolution of $T$ and $u$ to be the function $T\star_\gamma u\colon\C\to\C$ given by
\begin{equation}\label{conv}
 T{}\star_\gamma{}u(z) = \langle T, T^{\gamma}_{z}u \rangle,\qquad z\in\C.
 \end{equation}
 \end{definition}
\begin{proposition}\label{prop : conv oper}
 For every  $T\in \dhr$, the mapping  $u\to T\star_\gamma u$ \ is continuous from  $\hr$  into itself.
\end{proposition}
\begin{proof} Let $u\in\hr$. Since, by Proposition~\ref{prop : conti trans}, the series in \eqref{eq : serie trans 1} converges in $\hr$ it follows that
\begin{equation*}
T{}\star_\gamma{}u(z)=\sum_{n=0}^{+\infty}\frac{\langle T,B_{r}^{n}u\rangle}{\alpha_{rn}(\ga)} z^{rn},\qquad\text{for all $z\in\C$}.
\end{equation*}
Hence, $T{}\star_\gamma{}u\in\hr$. On the other hand, by the continuity of $T$, there exist $C$, $R'>0$ such that
\begin{equation}
|\langle T,\varphi\rangle|\leq C\|\varphi\|_{R'},\qquad\text{for all $\varphi\in\hr$}.
\end{equation}
Now, let $R>0$. Using Lemma~\ref{lem : ineq br} we conclude that
\begin{equation}
\left|\frac{\langle T,B_{r}^{n}u\rangle}{\alpha_{rn}(\ga)} z^{rn}\right|\leq C\frac{R^{nr}M^n(nr)!}{(R')^{nr}\alpha_{rn}(\ga)}\|u\|_{2R^\prime},
\end{equation}
for all $z\in\C$ such that $|z|\leq R$ and all $n\in\N$. The rest of the proof runs in a similar way as the proof of Proposition~\ref{prop : conti trans}.
\end{proof}
\begin{definition} If $T$, $S\in \hr$, the convolution $T{}\star_\gamma{}S$ is the element of $\dhr$ defined by
 \begin{equation*}
 \langle T{}\star_\gamma{}S, u\rangle = \langle T, S{}\star_\gamma{}u \rangle,\qquad u\in \hr.
 \end{equation*}
\end{definition}
In the following proposition, we give some algebraic properties of the generalized convolution. The proof follows from Theorem~\ref{th : pw}.
\begin{proposition}\label{prop : propri conv} Let $T,S,R\in\dhr$. Then
\begin{enumerate}\romanletters
\item $\mathcal{F}_{\gamma} (T\star_\gamma S)=\mathcal{F}_{\gamma}(T)\mathcal{F}_{\gamma}(S)$.
\item $T\star_\gamma S=S\star_\gamma T$.
\item $T\star_\gamma(S\star_\gamma R)=(T\star_\gamma S)\star_\gamma R$.
\item $T\star_\gamma \delta=T,$ where $\delta$ denotes the Dirac functional.
\item $B_{r}(T\star_\gamma S)=(B_{r}T)\star_\gamma S=T\star_\gamma(B_{r}S),$ where
$B_{r}$ is defined on $\mathcal{H}_{r}^{\prime}$ by transposition.
\end{enumerate}
\end{proposition}
\begin{proposition}Let $T\in\dhr$. If $T$ is nonzero, then the map $T\!\star_\gamma\colon u\mapsto T\star_\gamma u$ from $\hr$ into itself is surjective.
\end{proposition}
\begin{proof} From \textrm{Lemma}~23.31 and \textrm{Theorem}~26.3 of \cite{vogt}, we know that the present statement is equivalent to the two properties that %
 the dual map $(T\star_\gamma)^\prime\colon\dhr\to\dhr$ is injective and has closed image. The first condition follows from \textrm{Proposition}~\ref{prop : propri conv}. %
 Let us prove the second condition. Since $\dhr$ has the structure of a Montel $DF$-space, by \cite{destopo}*{\textrm{Theorem}~15.12}, it is enough to prove that %
 $(T\star_\gamma)^\prime$ has a sequentially closed image. Suppose that we have a sequence $(S_n)_n$ in $\dhr$ and $S\in\dhr$ such that $(S_n\star_\gamma T)\to S$
in $\dhr$. Then by \textrm{Theorem}~\ref{th : pw} $\mathcal{F}_{\gamma}(S_n)\mathcal{F}_{\gamma}(T)\to \mathcal{F}_{\gamma}(S)$ in $\e$. Since $\e$ is
continuously embedded in $\hr$, the convergence holds also uniformly over compact sets. Therefore, the entire function $\mathcal{F}_{\gamma}(S)$ vanishes at all the zeros of $\mathcal{F}_{\gamma}(T)$, with at least the same multiplicity. %
Hence, $f= \mathcal{F}_{\gamma}(S)/\mathcal{F}_{\gamma}(T)$ is an entire function. So, Lindelöf's theorem \cite{Ber2}*{§4.5.7} ensures that $f$ is of exponential type. %
Using Theorem~\ref{th : pw}, we conclude that there is $R\in\dhr$ such that $\F(R)=f$ and hence,
\begin{equation*}
\mathcal{F}_{\gamma}(R)\mathcal{F}_{\gamma}(T) = \mathcal{F}_{\gamma}(S).
\end{equation*}
Thus, $S = T\star_\gamma R$. This completes the proof.
\end{proof}
\section{ Chaotic  character of the generalized convolution~operators}\label{s5}
We begin by characterizing the continuous linear mappings from $\hr$ into itself that commute with generalized translation operators.
\begin{proposition}\label{prop : commut conv} If $\mathcal{L}$ is a continuous linear mapping from $\hr$ into itself, then following are equivalent:
\begin{enumerate}\romanletters
\item\label{v} The operator $\mathcal{L}$ commutes with the  operator $B_{r}$.
\item\label{i}The operator $\mathcal{L}$ commutes with $T^{\gamma}_{z}$ for all $z\in\C$.
\item\label{ii} There exists a unique $T\in\dhr$ such that
\begin{equation}\label{eq : l conv}
\mathcal{L} u=T\star_\gamma u, \qquad\text{for all $u\in\hr$}.
\end{equation}
\item\label{iv} There exists $\Phi(z)=\sum_{n=0}^{+\infty}\frac{b_{n}}{\alpha_{rn}}z^{rn}\in\e$ such that $\mathcal{L}=\Phi(B_{r})$. That is, for every $u\in\hr$,
\begin{equation}\label{eq : phi}
\mathcal{L}(u)(z)=\left[\Phi(B_{r})u\right](z)=\sum_{n=0}^{+\infty}\frac{b_{n}}{\alpha_{rn}}B_{r}^{n}u(z),\qquad z\in\C.
\end{equation}
Moreover, the series in~\eqref{eq : phi} converges in $\hr$.
\end{enumerate}
\end{proposition}
\begin{proof}
\eqref{v}$\Rightarrow$\eqref{i}: Since, by Proposition~\ref{prop : conti trans}, the series in \eqref{serie : trans} converges in $\hr$ and $\mathcal{L}$ is continuous on $\hr$ it follows that
\begin{equation*}
\mathcal{L}(T^{\ga}_{z}u)=\sum_{n=0}^{+\infty}
\frac{z^{rn}}{\alpha_{rn}(\gamma)}\mathcal{L}\left( B_{r}^{n}u\right)=\sum_{n=0}^{+\infty} \frac{z^{rn}}{\alpha_{rn}(\gamma)} B_{r}^{n}\mathcal{L}(u)=T^{\ga}_{z}\mathcal{L}(u).
\end{equation*}
\eqref{i}$\Rightarrow$\eqref{ii}: Suppose that $T\in\dhr$ is such that \eqref{eq : l conv} holds. Then, obviously, $T(u)=(\mathcal{L}u)(0)$ for all $u\in\hr$. Hence $T$ is unique. Conversely, the mapping $T\colon u\mapsto \mathcal{L}(u)(0)$ belong to $\dhr$ and, by Proposition~\ref{prop : prop trans} \eqref{a} and~\eqref{b}, for all $u\in\hr$ we have
\begin{equation*}
\mathcal{L}(u)(z)=[T_0^\gamma(\mathcal{L}(u)](z)=[T_z^\gamma(\mathcal{L}(u)](0)=\mathcal{L}(T_z^\gamma u)(0)=\langle T, T_z^\gamma u\rangle.
\end{equation*}
for all $z\in\C$. Thus, $\mathcal{L}(u)=T\star_\gamma u$.\\
\itemref{ii}$\Rightarrow$\itemref{iv}: Suppose that $T\in\dhr$ is such that \eqref{eq : l conv} holds. Since, by Proposition~\ref{prop : conti trans} the series in \eqref{serie : trans}
    converges in $\hr$ with respect to $w$, we have
\begin{equation}\label{1}
(\mathcal{L}u)(z)=\sum_{n=}^{+\infty}\frac{\langle T, w^{rn}\rangle}{\alpha_{rn(\gamma)}}B_r^nu(z),\qquad\text{for all $z\in\C$}.
\end{equation}
For $n\in\N$, set $b_n=\langle T, w^{rn}\rangle$. Using the continuity of $T$, we can see that there exist $C$, $a>0$ such that $|b_n|\leq C a^{rn}$ for all $n\in\N$. Hence, by Lemma~\ref{pr : est bn}, $\Phi(z)=\sum_{n=0}^{+\infty}\frac{b_{n}}{\alpha_{rn}}z^{rn}\in\e$. The convergence in $\hr$ of the series in \eqref{1} can be established similarly to the proof of Proposition~\ref{prop : conti trans} using Lemma~\ref{lem : ineq br}.\\
\eqref{iv}$\Rightarrow$\eqref{v}: Since $B_r$ is continuous on $\hr$ it follows that
\begin{equation*}
B_r(\mathcal{L}u)=B_r\left(\sum_{n=0}^{+\infty}\frac{b_{n}}{\alpha_{rn}}B_{r}^{n}u\right)=\sum_{n=0}^{+\infty}\frac{b_{n}}{\alpha_{rn}}B_{r}^{n+1}u=\mathcal{L}(B_ru),
\end{equation*}
for all $u\in\hr$. This completes the proof.
\end{proof}

One of the most important result in theory of hypercyclic and chaotic operator is the Godefroy-Shapiro criterion given by the following theorem.
\begin{theorem}[Godefroy-Shapiro]\label{th : ged-shap}(See \cite{linchao}*{Theorem 3.1 Page 69}) Let $X$ a be a Fréchet space and Let $T$ be an operator on $X$. Suppose that the subspaces
\begin{align}
X_0:=&\, \textrm{\em span}\{x\in X\, \mid Tx = \lambda x,\quad \textrm{for some $\lambda\in\C$ with $|\lambda|<1$}\},\label{set : X}\\
Y_0:=&\, \textrm{\em span}\{x\in X\, \mid Tx = \lambda x,\quad \textrm{for some $\lambda\in\C$ with $|\lambda|>1$}\}\label{set : Y}
\end{align}
are dense in $X$. Then $T$ is hypercyclic. If, moreover, the subspace
\begin{equation}\label{set : Z}
Z_0 := \textrm{\em span}\{x\in X\; \mid Tx =e^{i\pi\alpha} x ,\quad\textrm{for some $\alpha\in\Q$}\}
\end{equation}
is dense in $X$, then $T$ is chaotic.
\end{theorem}
An operator $\mathcal{L}$ from $\hr$ into itself that satisfies the conditions in the statement of Proposition~\ref{prop : commut conv} is called convolution operator. In the following theorem we obtain the chaos of the convolution operators associated with the operator $B_r$ on $\hr$.
\begin{theorem} If $\mathcal{L}$ is a convolution operator associated with the operator $B_r$ on $\hr$, and if $\mathcal{L}$ is not a scalar multiple of the identity, then it is a chaotic operator.
\end{theorem}
\begin{proof}By virtue of Proposition \ref{prop : commut conv}, we can find an entire function $\Phi(z)=\sum_{n=0}^{+\infty}\frac{b_n}{\alpha_{rn}(\gamma)}z^{rn}\in\e$, such that $\mathcal{L}=\Phi(B_{r})$. In particular, by Proposition~\ref{pr : jgamma egenfunction}, for every $\lambda\in\C$ we have
\begin{equation*}
\mathcal{L} j_{\gamma}(\lambda \cdot)=\Phi(B_{r})j_{\gamma}(\lambda \cdot)=\sum_{n=0}^{+\infty}\frac{b_n}{\alpha_{rn}(\gamma)}(-\lambda^r)^nj_\gamma(\lambda\cdot)=\Phi(e^{i\pi/r}\lambda)j_{\gamma}(\lambda \cdot),
\end{equation*}
wher $j_\gamma(\lambda\cdot)$ is given by \eqref{eq : jlambda}. To simplify, let $\Psi(\lambda)=\Phi(e^{i\pi/r}\lambda)$, for all $\lambda\in\C$. It follows that for every $\lambda\in\C$, the function %
$j_\gamma(\lambda \cdot)\colon z\mapsto j_{\gamma}(\lambda z)$ is an %
eigenfunction of $\mathcal{L}$ associated with the eigenvalue $\Psi(\lambda)$. Since $\mathcal{L}$ is not a scalar multiple of the identity, the function $\Psi$ is not constant. %
So, the sets $A=\left\{z\in\C\mid\left\vert \Psi(z)\right\vert <1\right\}$ and $B=\left\{z\in \C\mid \left\vert \Psi(z)\right\vert >1\right\}$ are open and nonempty. Hence,
according to  Corollary~\ref{cor : dens},
\begin{equation*}
\textrm{ span}\{f\in \hr\, \mid\, \mathcal{L}f = \Psi(\lambda) f,\quad \textrm{for some $\lambda\in A$}\},
\end{equation*}
and
\begin{equation*}
\textrm{ span}\{f\in \hr\, \mid\, \mathcal{L}f = \Psi(\lambda) f,\quad \textrm{for some $\lambda\in B$}\}
\end{equation*}
are dense in $\hr$. So, Godefroy-Shapiro subspaces $X_0$ and $Y_0$ given by \eqref{set : X} and \eqref{set : Y} are dense in $\hr$. Hence, by Theorem~\ref{th : ged-shap} the operator $\mathcal{L}$ is hypercyclic. On the other hand, using the same method as in \cite{linchao}*{Page~108},
we see that the set
$$\{\lambda\in\C\mid \Psi(\lambda)=e^{\alpha i\pi}, \textrm{ for some $\alpha\in\Q$}\}$$
 has an accumulation point. Hence, by Corollary~\ref{cor : dens}, we conclude that
\begin{equation*}
\mathrm{span}\{j_{\gamma}(\lambda\cdot)\mid\lambda\in \C,\quad \textrm{such that $\Psi(\lambda)=e^{\alpha i\pi}$ for some $\alpha\in\Q$}\}.
\end{equation*}
is dense in $\hr$. So the Godefroy-Shapiro subspace $Z_0$ given by \eqref{set : Z} is dense in $\hr$. Hence, Theorem~\ref{th : ged-shap} implies that $\mathcal{L}$ is chaotic.
\end{proof}
\begin{remark}Under the hypotheses of Theorem~\ref{th : ged-shap}, by Birkoff transitivity theorem \cite{bay}*{Theorem~1.2, Page~2}, the set of hypercyclic vector for $\mathcal{L}$ is a dense $G_\delta$ set of $\hr$ and by Herrero-Bourdon theorem \cite{linchao}*{Theorem 2.55}, there is a linear space $\mathcal{M}$ of $\hr$ such that every element of $\mathcal{M}\setminus\left\{  0\right\}  $ is hypercyclic for $\mathcal{L}$.
\end{remark}

\end{document}